\documentclass[11pt]{elsarticle}
\usepackage{amsmath,amssymb,amsthm}
\usepackage{mathrsfs}
\usepackage{graphicx}
\usepackage{tikz}
\usepackage{enumerate}
\usepackage{float}
\usepackage[normalem]{ulem}
\usetikzlibrary{positioning}
\usepackage{tikz-cd}

\usetikzlibrary{calc}
\theoremstyle{plain}

\theoremstyle{definition}
\newtheorem{theorem}{Theorem}[section]
\newtheorem{lemma}[theorem]{Lemma}
\newtheorem{proposition}[theorem]{Proposition}
\newtheorem{definition}[theorem]{Definition}
\newtheorem{remark}[theorem]{Remark}
\newtheorem{corollary}[theorem]{Corollary}
\newtheorem{example}[theorem]{Example}

\makeatletter

\newcommand{\id}{\mathrm{id}}

\newcommand{\Rmnum}[1]{\expandafter@slowromancap\romannumeral #1@}
\makeatother
\makeatletter
\def\ps@pprintTitle{%
  \let\@oddhead\@empty
  \let\@evenhead\@empty
  \def\@oddfoot{\reset@font\hfil\thepage\hfil}
  \let\@evenfoot\@oddfoot
}\makeatother

\begin{document}
\begin{frontmatter}

\title{On $\mathcal{U}_c$-admitting spaces and $\mathcal{U}_c$-completions of $T_0$ spaces\tnoteref{t1}}
\tnotetext[t1]{Supported by NSF of China (No.12231007).}

\author[1]{Kangming Lu}
\ead{481187071@qq.com}

\author[1]{Qingguo Li\corref{a1}}
\ead{liqingguoli@aliyun.com}
\author[2]{Shengwen Wang}
\ead{wangshw7201@163.com}
\author[1]{Rongqi Xiao}
\ead{rongqixiao2024@163.com}
\cortext[a1]{Corresponding author.}
\address[1]{School of Mathematics, Hunan University, Changsha, Hunan, 410082, China}
\address[2]{School of Mathematics and Statistics, Liupanshui Normal University, Guizhou 553004,China}

\begin{abstract}
In this paper, we introduce a new class of spaces, termed $\mathcal{U}_c$-admitting spaces, and investigate their relationships with $\mathcal{U}_s$-admitting spaces and well-filtered spaces. Adopting the perspective of H-sobriety, we show that every \(\mathcal{U}_c\)-admitting space is H-sober, so that $\mathcal{U}_c$-admitting spaces arise as a special case of H-sober spaces.
 Consequently, the category \(\mathbf{U_c\text{-}Sob}\) of $\mathcal{U}_c$-admitting spaces and continuous maps is complete and forms a reflective subcategory of \(\mathbf{Top}_0\). In addition, we prove that a \(T_0\) space $X$ is \(\mathcal{U}_c\)-admitting if and only if its Smyth power space is \(\mathcal{U}_c\)-admitting. Finally, we present the standard \(\mathcal{U}_c\)-completion of an arbitrary \(T_0\) space.

\end{abstract}

\begin{keyword}
well-filtered spaces, \(\mathcal{U}_c\)-admitting spaces, Smyth power spaces, completions\\[0.5ex]
\text{MSC (2020):} 54D99, 54B20, 06B35, 06F30.
\end{keyword}

\end{frontmatter}
\section{Introduction}
Domain theory was introduced by Dana Scott in the late 1960s as a mathematical foundation for the denotational semantics of programming languages. It provides a framework for studying directed-complete partial orders, i.e., partially ordered sets in which every directed subset has a supremum, and has since become a central area of research in theoretical computer science and non-Hausdorff topology. Among its many topological manifestations, \(d\)-spaces, well-filtered spaces and sober spaces are recognized as three of the most important classes, each capturing a distinct aspect of approximation and convergence \cite{goubault2013,gierz2003,XuZhao2021}.

The relationships among these three classes are well understood: every sober space is well-filtered, and every well-filtered space is a \(d\)-space; neither implication can be reversed in general \cite{Shen2019}. In recent years, much effort has been devoted to refining this hierarchy and to understanding the exact boundary between these classes. In particular, Heckmann introduced the notion of \(\mathcal{U}_s\)-admitting spaces \cite{Heckmann1992}, and Shen later observed that they lie strictly between \(d\)-spaces and well-filtered spaces \cite{Shen2025}. Moreover, Shen proved that the category of \(\mathcal{U}_s\)-admitting spaces with continuous maps is a reflective subcategory of \(\mathbf{Top}_0\), and established several equivalent characterizations of these spaces.

However, the existing framework leaves a natural question open: can the gap between \(\mathcal{U}_s\)-admitting spaces and well-filtered spaces be further refined? In this paper, we give a positive answer by introducing a new class of spaces, which we call \(\mathcal{U}_c\)-admitting spaces. The definition of \(\mathcal{U}_c\)-admitting spaces is motivated by a natural weakening of the compactness condition in the definition of well-filtered spaces: while well-filteredness requires the existence of a compact saturated set inside every open neighbourhood, \(\mathcal{U}_c\)-admittingness only requires a weakly strongly compact saturated set. This modification, though seemingly minor, gives rise to a rich theory.

In domain theory, various completions of \(T_0\) spaces have been extensively studied and occupy a central position. These range from the classical sobrification, to the more recent well-filterification, \(D\)-completion, and \(\mathcal{U}_s\)-completion \cite{gierz2003,Shen2019,Li2017,Xu2020,Keimel2009}. In \cite{Miao2024}, sufficient conditions for the coincidence of some of these completions were provided. A common feature of these constructions is that each completion is determined by a corresponding family of subsets: the \(D\)-completion is determined by tapered sets, the well-filterification by well-filtered determined sets, and the \(\mathcal{U}_s\)-completion by \(\mathcal{U}_s\)-admitting determined sets. In this paper, we introduce the notion of \(\mathcal{U}_c\)-admitting determined sets, and use it to construct the \(\mathcal{U}_c\)-completion of arbitrary \(T_0\) spaces.

The main results of this paper are as follows.

\begin{enumerate}
    \item We establish the strict inclusions
    \[
    \text{well-filtered} \subsetneq \mathcal{U}_c\text{-admitting} \subsetneq \mathcal{U}_s\text{-admitting},
    \]
    thereby refining the known hierarchy between \(d\)-spaces and well-filtered spaces.
    
    \item We prove that in a countable \(T_0\) space, \(\mathcal{U}_c\)-admitting coincides with well-filteredness, and that \(\mathcal{U}_c\)-admitting together with local countable compactness (or together with first-countability) implies sobriety.
    
    \item Adopting the H-sober framework of Xu, we show that \(\mathcal{U}_c\)-sets form a class of \(H\)-sets; and consequently, the category \(\mathbf{U}_c\text{-Sob}\) of \(\mathcal{U}_c\)-admitting spaces is complete and reflective in \(\mathbf{Top}_0\). We also prove that a \(T_0\) space is \(\mathcal{U}_c\)-admitting if and only if its Smyth power space is \(\mathcal{U}_c\)-admitting.
    
    \item Finally, we construct the standard \(\mathcal{U}_c\)-completion of a \(T_0\) space, which provides a universal way to embed any \(T_0\) space into a \(\mathcal{U}_c\)-admitting space.
\end{enumerate}

The paper is organized as follows. Section 2 recalls the necessary preliminaries, including \(d\)-spaces, well-filtered spaces, sober spaces, $\mathcal{U}_s$-admitting spaces and the Smyth power space. Section 3 introduces \(\mathcal{U}_c\)-sets and \(\mathcal{U}_c\)-admitting spaces, and investigates their basic properties and characterizations. Section 4 places \(\mathcal{U}_c\)-admitting spaces in the H-sober framework and establishes their categorical properties. Section 5 constructs the \(\mathcal{U}_c\)-completion of a \(T_0\) space.

\section{Preliminaries}

We introduce some basic concepts and results to be used later. For more details, we refer the readers to \cite{gierz2003,goubault2013}.

Let \( P \) be a poset. For any subset \( A \) of a poset \( P \), we use the standard notation
\[
\mathord{\uparrow} A = \{ y \in P : \exists x \in A, x \leq y \}; \quad 
\mathord{\downarrow} A = \{ y \in P : \exists x \in A, y \leq x \}.
\]

For each \( x \in X \), we write \( \mathord{\uparrow} x = \mathord{\uparrow} \{ x \} \) and \( \mathord{\downarrow} x = \mathord{\downarrow} \{ x \} \). 
A subset \( A \) of \( P \) is called a lower (resp., upper) set if \( A = \mathord{\downarrow} A \) (resp., \( A = \mathord{\uparrow} A \)).

A nonempty subset \( D \) of \( P \) is directed if every two elements in \( D \) have an upper bound in \( D \). \( P \) is called a directed complete poset, or a dcpo for short, if for any directed subset \( D \subseteq P \), the supremum \( \bigvee D \) of \( D \) exists. 

A subset \( U \) of \( P \) is \textit{Scott open} if (i) \( U = \mathord{\uparrow} U \) and (ii) for any directed subset \( D \) of \( P \) for which \( \bigvee D \) exists, \( \bigvee D \in U \) implies \( D \cap U \neq \emptyset \). All Scott open subsets of \( P \) form a topology, called the \textit{Scott topology} on \( P \), denoted by \( \sigma(P) \). The space \( \Sigma P = (P, \sigma(P)) \) is called the \textit{Scott space} of \( P \).

Let \( X \) be a \( T_0 \) space. A subset \( A \) of \( X \) is called saturated if \( A \) equals the intersection of all open sets containing it. The specialization order \( \leq \) on \( X \) is defined by \( x \leq y \) if and only if \( x \in \overline{\{y\}}= \mathord{\downarrow} y \), where $\overline{\{y\}}$ is the closure of $\{y\}$. It is important to note that a subset \( A \) of \( X \) is saturated if and only if \( A = \mathord{\uparrow} A \) with respect to the specialization order.

A set $A$ is called compact if for every open cover $\{U_i\}_{i \in I}$ of $A$ (i.e., $A \subseteq \bigcup_{i \in I} U_i$ where each $U_i$ is open), there exists a finite subset $J \subseteq I$ such that $A \subseteq \bigcup_{i \in J} U_i$. The closed subset of a compact set is also compact, and the continuous image of a compact set is also compact. A topological space $X$ is called $A_1$ (or first-countable) if every point $x$ has a countable neighborhood base; it is called locally compact if for every point $x$ and every open neighborhood $U$ of $x$, there exists a compact set $K$ such that $x \in \operatorname{int}(K) \subseteq K \subseteq U$; and it is called core compact if its open set lattice $\mathcal{O}(X)$ is continuous. Every locally compact space is core compact \cite{goubault2013}.

For a \( T_0 \) space \( X \), \( A \subseteq X \), and \( \mathcal{F} \subseteq 2^X \), we shall use the following notations.
\begin{itemize}
    \item \( \mathcal{O}(X) \), the family of all open subsets of \( X \);
    \item \( \mathcal{S}(X) \),  the family of all singleton set;
    \item \( \mathcal{C}(X) \), the family of all closed subsets of \( X \);
    \item \( \mathcal{Q}(X) \), the family of all compact saturated subsets of \( X \);
    \item \( \mathcal{Q}^*(X) \), \( \mathcal{Q}(X) \setminus \{\emptyset\} \);
    \item \( \mathcal{M}(\mathcal{F}) \), the family \( \{C \in \mathcal{C}(X) : \forall F \in \mathcal{F}, C \cap F \neq \emptyset\} \);
    \item \( m(\mathcal{F}) \), the family of all minimal members in \( (\mathcal{M}(\mathcal{F}), \subseteq) \).
\end{itemize}

\begin{definition}(\cite{gierz2003})
Let \( X \) be a \( T_0 \) space. A nonempty subset \( A \) of \( X \) is called \emph{irreducible} if for any closed sets \( F_1, F_2 \) of \( X \), \( A \subseteq F_1 \cup F_2 \) implies \( A \subseteq F_1 \) or \( A \subseteq F_2 \). The space \( X \) is called \emph{sober} if every irreducible closed subset of \( X \) is the closure of a (unique) point.
\end{definition}

\begin{lemma}(\cite[section2.4]{Heckmann2013})
Let \( X \) be a \( T_0 \) space and \( Y \) a subspace of \( X \). The following statements are equivalent for a subset \( A \subseteq Y \):
\begin{enumerate}
    \item \( A \) is an irreducible subset of \( Y \).
    \item \( A \) is an irreducible subset of \( X \).
    \item \( \overline{A} \) is an irreducible subset of \( X \).
\end{enumerate}
\end{lemma}

\begin{definition}(\cite{gierz2003,Shen2019})
A \( T_0 \) space \( X \) is called well-filtered if for any filtered family \( \mathcal{F} \) of compact saturated sets in \( X \) and \( U \in \mathcal{O}(X) \), \( \bigcap \mathcal{F} \subseteq U \) implies \( K \subseteq U \) for some \( K \in \mathcal{F} \).

A set $A$ is called a KF-set if there exists filtered family \( \mathcal{F} \) of compact saturated sets such that \( \overline{A} \in m(\mathcal{F})\). $X$ being well-filtered is equivalent to  the statement that the closure of every KF-set  is the closure of a (unique) point.
\end{definition}

\begin{definition}(\cite{gierz2003})
A \( T_0 \) space \( X \) is called a \( d \)-space if \( X \) is a dcpo and every open subset of \( X \) is Scott open in the specialization order. And $X$ being a $d$-space is equivalent to  the closure of every directed set is the closure of a (unique) point.
\end{definition}

\begin{proposition}(\cite[Corollary2.4, Lemma2.9]{Shen2019})
Let $X$ be a \(T_0\) space and $A$ be a subset of $X$. Then $A$ is a directed subset (with respect
to the specialization order ) implying that A is a KF-set; and A is a KF-set, implying that A is an
irreducible subset. Let \(\mathcal{D}(X)\) denote the family of all directed sets, \(\mathcal{KF}(X)\) the family of all KF sets, and \(\mathcal{IRR}(X)\) the family of all irreducible subsets of $X$.Then
\[
\mathcal{D}(X) \subseteq \mathcal{KF}(X) \subseteq \mathcal{IRR}(X)
\]
\[
\text{soberiety} \Longrightarrow \text{well-filteredness} \Longrightarrow d\text{-space}
\]
\end{proposition}

 We will use the following result, derived from the topological Rudin lemma (see Lemma 3.1 in \cite{Heckmann2013}).
  
\begin{lemma}
  
Let \( X \) be a \( T_0 \) space, \( C \) a closed subset of \( X \) and \( \mathcal{F} \) a filtered family of compact saturated sets in \( X \). Every closed set \( C \in \mathcal{M}(\mathcal{F}) \) contains a subset \( A \in m(\mathcal{F}) \). Every member in \( m(\mathcal{F}) \) is irreducible.
\end{lemma}

\begin{definition}\label{a3}(\cite{goubault2013})
A \emph{retract} of a topological space $Y$ is a  topological space $X$ such that there are two continuous maps $s: X\rightarrow Y$(section) and $r:Y\rightarrow X$(retraction) such that $r\circ s=\id_X $. 
\end{definition}

The following result is well-known (see \cite{gierz2003,Kou2001,Lawson2020,Xu2020}).

\begin{lemma}
For a \( T_0 \) space \( X \), the following four conditions are equivalent:
\begin{enumerate}
    \item \( X \) is locally compact and sober.
    \item \( X \) is locally compact and well-filtered.
    \item \( X \) is core compact and sober.
    \item \(X \) is core compact and well-filtered.
\end{enumerate}
\end{lemma}

From \cite[Theorem 4.2]{XuShen2020}, we can also obtain the following lemma.
\begin{lemma}
Every $A_1$ well-filtered $T_0$ space $X$ is sober.
\end{lemma}
For a topological space \( X \), \( \mathcal{G} \subseteq 2^X \) and \( A \subseteq X \), let \( \Diamond_G A = \{G \in \mathcal{G} : G \cap A \neq \emptyset\} \) and \( \Box_G A = \{G \in \mathcal{G} : G \subseteq A\}\). The symbols \( \Diamond_G A \) and \( \Box_G A \) will be simply written as \( \Diamond A \) and \( \Box A \) respectively if there is no confusion. The \textbf{lower Vietoris topology} on \( \mathcal{G} \) is the topology that has \( \{\Diamond U : U \in \mathcal{O}(X)\} \) as a subbase, and the resulting space is denoted by \( P_H(\mathcal{G}) \). If \( \mathcal{G} \subseteq \operatorname{Irr}(X) \), then \( \{\Diamond_G U : U \in \mathcal{O}(X)\} \) is a topology on \( \mathcal{G} \). The space \( P_H(\Gamma(X) \setminus \{\emptyset\}) \) is called the \textbf{Hoare power space} or \textbf{lower space} of \( X \) and is denoted by \( P_H(X) \) for short \cite{Schalk1993}. Clearly, \( P_H(X) = (\Gamma(X) \setminus \{\emptyset\}, v(\Gamma(X) \setminus \{\emptyset\})) \). So \( P_H(X) \) is always sober by \cite[Corollary 4.10]{Zhao2015}. The \textbf{upper Vietoris topology} on \( \mathcal{G} \) is the topology that has \( \{\Box_G U : U \in \mathcal{O}(X)\} \) as a base, and the resulting space is denoted by \( P_S(\mathcal{G}) \).When \(\mathcal{G}\)  is \( \mathcal{Q}(X) \),we usually denote it \( P_s(X) \).

\begin{proposition}(\cite[Theorem 5.3]{Xu2020} or \cite[Theorem 4]{Xu2021})
For a \(T_0\) space $X$ and its Smyth power space \( P_s(X) \),the following conditions are equivalent.
\begin{enumerate}
 \item \( X \) is well-filtered.
 \item \(P_s(X)\) is well-filtered.
 \item \(P_s(X)\) is a d-space.
 \end{enumerate}
 \end{proposition}

Next, we introduce the notion of strongly compact sets, \(\mathcal{U}_s\)-sets and \(\mathcal{U}_s\)-admitting spaces, whose definition is taken from \cite{Heckmann1992}.
\begin{definition}
    A subset \( K \) of a \( T_0 \) space \( X \) is \emph{strongly compact} if for each open set \( U \) with \( K \subseteq U \), there is a finite set \( F \subseteq X \) such that \( K \subseteq \mathord{\uparrow} F \subseteq U \). Let \( \mathcal{Q}_s(X) = \{\mathord{\uparrow} K \subseteq X: \text{$K$ is strongly compact} \} \). In $T_1$ spaces, strongly compact spaces are exactly  finite sets.

    A nonempty subset \( A \) of a \( T_0 \) space \( X \) is called a \( \mathcal{U}_s \)-set if there is a filtered subfamily \( \mathcal{F} \) of \( \mathcal{Q}_s(X) \) such that \( \overline{A} \in m(\mathcal{F}) \). Let \(\mathcal{U}_s(X)\) denote the family of all \(\mathcal{U}_s\)-sets in the \(T_0\) space \(X\).

    A $T_0$ space is called \(\mathcal{U}_s\)-admitting space if for any \( \mathcal{U}_s \) set $A$, there exists a point $x$ such that \(\overline{A}=\overline{\{x\}}\). Every $T_1$ space is a \(\mathcal{U}_s\)-admitting space.
    
\end{definition}

 By the definition of \(\mathcal{U}_s \)-set and \(\mathcal{U}_s\)-admitting space, it is easy to reach the following conclusion.
\begin{remark}
In a $T_0$ space X.
\begin{enumerate}
    \item \(\mathcal{D}(X) \subseteq \mathcal{U}_s(X) \subseteq \mathcal{KF}(X) \subseteq \mathcal{IRR}(X).\)
    \item \(\text{soberiety} \Longrightarrow \text{well-filteredness} \Longrightarrow \text{\(\mathcal{U}_s\)-admitting space}\Longrightarrow d\text{-space}.\)
 \end{enumerate}
\end{remark}

In the following, we provide two examples to illustrate that \(\mathcal{U}_s\)-sets are distinct from both directed sets and KF-sets, and that \(\mathcal{U}_s\)-admitting spaces are distinct from well-filtered spaces.
\begin{example}
The set of natural numbers $\mathbb{N}$:
\begin{center}
\begin{tikzpicture}
    \foreach \x in {1,...,8} {
        \filldraw (\x,0) circle (2.5pt);
        \node[below] at (\x,-0.3) {$\x$};
    }
    \node at (9,0) {$\cdots$};
\end{tikzpicture}
\\Figure 1:$\mathbb{N}_\text{cof}$
\label{fig:example1}
\end{center}

Let $\mathbb{N}$ be equipped with the cofinite topology, denoted by $\mathbb{N}_\text{cof}$. Then $\mathbb{N}_\text{cof}$ satisfies the following properties:\\
\(\bullet\) Every subset of $N$ is compact.\\
\(\bullet\) The specialization order of $\mathbb{N}_\text{cof}$ is the discrete order, and $\mathbb{N}_\text{cof}$ is $T_1$.\\
\(\bullet\)  The closure of a finite set is itself, and the closure of an infinite set is $\mathbb{N}$.

Take an infinite set $A$. We prove that $A$ is a KF-set but not a $\mathcal{U}_s$-set.

    Let $\mathcal{A} = \{ \mathbb{N} \setminus F : F \text{ is finite} \}$. For every $K \in \mathcal{A}$, since $\mathbb{N}_{\text{cof}}$ is $T_1$, $K$ is saturated, hence compact saturated. For every $K_1, K_2 \in \mathcal{A}$, $\mathbb{N}  \setminus (K_1 \cap K_2)$ is finite, thus $K_1 \cap K_2 \in \mathcal{A}$, which shows that $\mathcal{A}$ is filtered. Since $\overline{A} = \mathbb{N}$, it follows that $\overline{A} \in \mathcal{M}(\mathcal{A})$,   and a finite set cannot intersect every element of $\mathcal{A}$, so $A \in m(\mathcal{A})$. This indicates $A$ is a KF-set.
    
In this space, strongly compact saturated sets are exactly finite sets. Suppose there exists a filtered family $\mathcal{F}$ such that $A = m(\mathcal{F})$, the elements of $\mathcal{F}$ are finite. Since $\mathcal{F}$ is filtered and $\mathbb{N}_{\text{cof}}$ is compact, $\bigcap \mathcal{F} \neq \emptyset$. Let $B = \bigcap \mathcal{F}$. Then $\overline{B} = B$ and $B \in \mathcal{M}(\mathcal{F})$, contradicting the minimality of $A$. Hence, $A$ is not a $\mathcal{U}_s$-set.

This proves that a KF set is not necessarily a $\mathcal{U}_s$-set. In this space, we can easily obtain that every $\mathcal{U}_s$-set is a singleton; as a result, $\mathbb{N}_{\text{cof}}$ is a $\mathcal{U}_s$-admitting space. However, the filtered family $\mathcal{A}$ defined above consists of compact saturated sets whose intersection is empty. Therefore, $\mathbb{N}_{\text{cof}}$ is not well-filtered, since in a well-filtered space the intersection of any filtered family of compact saturated sets is nonempty. Thus a \(\mathcal{U}_s\)-admitting space is not necessarily well-filtered.
\end{example}

\begin{example}
Let $X$ be a natural number chain with a top element, together with a side chain $b$.

\begin{center}
\begin{tikzpicture}[node distance=0.5cm]
    % 最大元 a
    \node (a) {$a$};
    
    % 省略号（必须用 $...$ 包裹）
    \node (dots) [below=of a] {$\vdots$};
    
    % 自然数链
    \node (3) [below=of dots] {$3$};
    \node (2) [below=of 3] {$2$};
    \node (1) [below=of 2] {$1$};
    \node (0) [below=of 1] {$0$};
    
    % 侧链 b
    \node (b) [below right=1cm of a] {$b$};
    
    % 连线
    \draw (0) -- (1);
    \draw (1) -- (2);
    \draw (2) -- (3);
    \draw (3) -- (dots);
    \draw (dots) -- (a);
    \draw (b) -- (a);
\end{tikzpicture}\\
Figue 2
\label{fig:example2}
\end{center}

That is,  \( X = \mathbb{N} \cup \{a, b\} \). We define a partial order \( \le \) on \( X \).  \( x \le y \) if \( x, y \in \mathbb{N} \) and \( x \le y \) in the usual order of natural numbers, or $y=a$, no other nontrivial relations hold; in particular, \( a \) is the greatest element of \( X \), and \( b \) is only comparable with \( a \).
 Now, we endow it with the Scott topology. Observe that $X \setminus \{a\}$ is not a directed set, since $1$ and $b$ have no common upper bound. However, $X \setminus \{a\}$ is a $\mathcal{U}_s$-set. $\overline{X \setminus \{a\}} = X$, since the natural number chain is a directed set whose supremum is $a$ and Scott closed sets are closed under directed sets; let $\mathcal{F} = \{\mathord{\uparrow} x : x \in X\}$, $X \in m(\mathcal{F})$. This shows that $\mathcal{U}_s$-sets are not necessarily directed sets.
\end{example}

By means of local finitary compactness, the relationship between \(\mathcal{U}_s\)-admitting spaces and sober spaces was established in \cite{Shen2025}.
\begin{definition}(\cite[Exercise 5.1.42]{goubault2013})
    A topological space \( X \) is \emph{locally finitary compact} if for every \( x \in X \) and every open neighborhood \( U \) of \( x \), there is a finite subset \( F \) of \( X \) such that \( x \in \operatorname{int}(\mathord{\uparrow} F) \subseteq \mathord{\uparrow} F \subseteq U \).
\end{definition}
\begin{remark}(\cite[Proposition 3.6]{Shen2025})
    Every locally finitary compact $\mathcal{U}_s$-admitting space is sober.

\end{remark}

    It is natural to ask whether $\mathcal{U}_s$-admitting spaces have properties like those in 2.8 or 2.9. Unfortunately, they do not. Recall Example $\ref{fig:example1}$, the space $\mathbb{N}_\text{cof}$ is $\mathcal{U}_s$-admitting; it is not well-filtered, hence not sober. But it is locally compact, core compact, and $A_1$. We give the proof below.

    For any $x \in \mathbb{N}$ and any open neighborhood $U$ of $x$, let $K = U$. Then $K$ is compact and satisfies $x \in \operatorname{int}(K) \subseteq K \subseteq U$. Therefore, $\mathbb{N}_{\text{cof}}$ is locally compact; consequently, it is core compact.

     For any $n \in \mathbb{N}$, define $U_{n}^{k} = \mathbb{N} \setminus \{1, 2, \dots, n-1, n+1, \dots, n+k\}$ and let $\mathcal{A} = \{ U_{n}^{k} : k = 1, 2, 3, \dots \}$. Then $\mathcal{A}$ is a countable neighborhood base of $n$. Indeed, for any open neighborhood $U$ of $n$, $\mathbb{N} \setminus U$ is finite, say $\{x_1, \dots, x_m\}$ with $x_1 < \cdots < x_m$, it follows $U_{n}^{x_m - n} \subseteq U$. Thus $\mathbb{N}_{\text{cof}}$ is $A_1$.

\section{\(\mathcal{U}_c\)-admitting space} 

In this section, we give the definition of \(\mathcal{U}_c\)-admitting space. By counterexamples, we prove that \(\mathcal{U}_c\)-admitting space is different from \(\mathcal{U}_s\)-admitting space and well-filtered space, and then investigate some properties of this new space.

\begin{definition}
    A  subset \( K \) of a \( T_0 \) space \( X \) is \emph{weakly strongly compact} if (i)\(K\) is compact and (ii)for each open set \( U \) with \( K \subseteq U \), there is a countable set \( C \subseteq X \) such that \( K \subseteq \mathord{\uparrow} C \subseteq U \).
    Let \( \mathcal{Q}_c(X) = \{ \mathord{\uparrow} K \subseteq X: \text{$K$ is weakly strongly compact} \} \).  
\end{definition}

From the definition of weak strong compactness, we immediately obtain the following five propositions.  
\begin{proposition}
    In $T_1$ spaces, weakly strongly compact subsets are exactly countable and compact subsets.
    \begin{proof}
    In a $T_1$ space $X$, we assume $K$ is a weakly strongly compact subset, the specialization order of $X$ is discrete, so $\mathord{\uparrow} C = C$ for every subset $C \subseteq X$. Consequently, if $K \subseteq C$ and $C$ is countable, then $K$ is countable as well. Conversely, if $K$ is countable and compact, then for any open set $U$ with $K \subseteq U$, we may set $C = K$, which satisfies the condition for weak strong compactness. 
    \end{proof}
\end{proposition}

\begin{proposition}
    Let \( X \) be a \( T_0 \) space and \( K \subseteq X \). Then \( K \) is weakly strongly compact if and only if \( \mathord{\uparrow} K \) is weakly strongly compact.
\end{proposition}

\begin{proposition}
    If $A$ and $B$ are weakly strongly compact, then $A \cup B$ is weakly strongly compact.
\end{proposition}
\begin{proof}
    Let $U$ be an open set such that $A \cup B \subseteq U$. Then $A \subseteq U$ and $B \subseteq U$. By weak strong compactness of $A$ and $B$, there exist countable sets $E, C \subseteq X$ such that
\[
A \subseteq \mathord{\uparrow} E \subseteq U \quad \text{and} \quad B \subseteq \mathord{\uparrow} C \subseteq U.
\]
Then
\[
A \cup B \subseteq \mathord{\uparrow} E \cup \mathord{\uparrow} C = \mathord{\uparrow}(E \cup C) \subseteq U.
\]
Since $E \cup C$ is countable, $A \cup B$ is weakly strongly compact.
\end{proof}

    \begin{proposition}
 The continuous image of a weakly strongly compact set is weakly strongly compact.
 \begin{proof}
     Suppose $f: X \to Y$ is continuous and $K \subseteq X$ is weakly strongly compact. For any open set $U \subseteq Y$ containing $f(K)$, we have $K \subseteq f^{-1}(U)$. Since $f$ is continuous, $f^{-1}(U)$ is open in $X$. By weak strong compactness of $K$, there exists a countable set $C \subseteq X$ such that $K \subseteq \mathord{\uparrow} C \subseteq f^{-1}(U)$. Since $f$ is order-preserving with respect to the specialization order, $\mathord{\uparrow} f(\mathord{\uparrow} C) = \mathord{\uparrow} f(C)$. Thus we obtain
\[
f(K) \subseteq \mathord{\uparrow} f(\mathord{\uparrow} C) = \mathord{\uparrow} f(C) \subseteq U,
\]
with $f(C)$ being countable. Hence,  $f(K)$ is weakly strongly compact.
 \end{proof}
        
    \end{proposition}

\begin{proposition}
    Let \( X \) be a \( T_0 \) space, \( K \) a weakly strongly compact subset of \( X \), and \( A \) a closed subset of \( X \). Then \( K \cap A \) is weakly strongly compact.
\end{proposition}
\begin{proof}
    Let \( U \) be an open set containing \( K \cap A \). Then
\[
K \subseteq U \cup (X \setminus A).
\]
Since \( A \) is closed, \( U \cup (X \setminus A) \) is open. Thus, there exists a countable set \( C \subseteq X \) such that
\[
K \subseteq \mathord{\uparrow} C \subseteq U \cup (X \setminus A).
\]

We claim that \( \mathord{\uparrow} C \cap A \subseteq \mathord{\uparrow} (C \cap A) \). Indeed, take \( x \in \mathord{\uparrow} C \cap A \). Then there exists \( y \in C \) with \( y \le x \). Since \( A \) is closed, \( A \) is a lower set and \( y \in A \). So \( y \in C \cap A \), it follows that \( x \in \mathord{\uparrow} (C \cap A) \). It means that
\[
K \cap A \subseteq \mathord{\uparrow} C \cap A \subseteq \mathord{\uparrow} (C \cap A) \subseteq U.
\]
Therefore, \( K \cap A \) is weakly strongly compact.
\end{proof}

We now propose two new notions, namely \(\mathcal{U}_c\)-sets and \(\mathcal{U}_c\)-admitting spaces, with the latter serving as the central subject of our investigation.
\begin{definition}
    A nonempty subset \( A \) of a \( T_0 \) space \( X \) is called a \(\mathcal{U}_c\)-set if there exists a filtered subfamily \(\mathcal{F} \subseteq \mathcal{Q}_c(X)\) such that \(\overline{A} \in m(\mathcal{F})\). Let \(\mathcal{U}_c(X)\) denote the family of all \(\mathcal{U}_c\)-sets in the \(T_0\) space \(X\).
\end{definition}

\begin{definition}
   A \( T_0 \) space \( X \) is called a \(\mathcal{U}_c\)-admitting space if for any \(\mathcal{U}_c\)-set \( A \subseteq X \), there exists a point \( x \in X \) such that
$\overline{A} = \overline{\{x\}}$.
\end{definition}

\begin{proposition}
    Every directed set is a $\mathcal{U}_s$-set, every $\mathcal{U}_s$-set is a $\mathcal{U}_c$-set, every $\mathcal{U}_c$-set is a KF-set, and every KF-set is an irreducible set. 
\end{proposition}
\begin{proof}
    By Remark 2.12, we only need to prove that every $\mathcal{U}_s$-set is a $\mathcal{U}_c$-set and that every $\mathcal{U}_c$-set is a KF-set. Suppose $X$ is a $T_0$ space and $A \subseteq X$. If $A$ is a $\mathcal{U}_s$-set, then there exists a filtered subfamily $\mathcal{F} \subseteq \mathcal{Q}_s(X)$ such that $\overline{A} \in m(\mathcal{F})$. Obviously, $\mathcal{Q}_s(X) \subseteq \mathcal{Q}_c(X)$, it follows that $A$ is a $\mathcal{U}_c$-set. If $A$ is a $\mathcal{U}_c$-set, then there exists a filtered subfamily $\mathcal{F} \subseteq \mathcal{Q}_c(X)$ such that $\overline{A} \in m(\mathcal{F})$. Since $\mathcal{Q}_c(X) \subseteq \mathcal{Q}(X)$, we conclude that $A$ is a KF-set.
\end{proof}
Thus \(\mathcal{D}(X) \subseteq \mathcal{U}_s(X) \subseteq \mathcal{U}_c(X) \subseteq \mathcal{KF}(X) \subseteq \mathcal{IRR}(X).\)

\begin{corollary}
\(\text{soberiety} \Longrightarrow \text{well-filteredness} \Longrightarrow \text{\(\mathcal{U}_c\)-admitting space}\Longrightarrow \text{\(\mathcal{U}_s\)-admitting space}\Longrightarrow d\text{-space}.\)   
\end{corollary}

To show that \(\mathcal{U}_c\)-admitting spaces lie strictly between \(\mathcal{U}_s\)-admitting spaces and well-filtered spaces, we need to provide two counterexamples. We first construct one example to show that a \(\mathcal{U}_s\)-admitting space need not be \(\mathcal{U}_c\)-admitting; the other will be given in \ref{ex2}.
\begin{example}
    Recall \ref{fig:example1} and the proof therein. 
The infinite set $A$ is a $\mathcal{U}_c$-set but not a $\mathcal{U}_s$-set. 
Thus $\mathbb{N}_{\mathrm{cof}}$ is a $\mathcal{U}_s$-admitting space which is not $\mathcal{U}_c$-admitting, 
because $\overline{A} = \mathbb{N}$ and $\mathbb{N}$ is not the closure of any singleton. 
Since $\mathbb{N}_{\mathrm{cof}}$ is $T_1$, this also shows that neither 
$\mathcal{U}_s$-admitting spaces nor $T_1$ spaces imply being a $\mathcal{U}_c$-admitting space.

\end{example}

\begin{proposition}
    The continuous image of a $\mathcal{U}_c$-set is a $\mathcal{U}_c$-set.
\end{proposition}
\begin{proof}
     Suppose $X$ and $Y$ are $T_0$ spaces and $f: X \to Y$ is continuous. Let $A \subseteq X$ be a $\mathcal{U}_c$-set. Then there exists a filtered family $\mathcal{F} \subseteq \mathcal{Q}_c(X)$ such that $\overline{A} \in m(\mathcal{F})$.

Claim 1:
$\mathcal{G} := \left\{ \mathord{\uparrow} f\big( A_i \cap \overline{A} \big) : A_i \in \mathcal{F} \right\}$
is contained in $\mathcal{Q}_c(Y)$. 

By Proposition 3.6, the set \(A_i \cap \overline{A}\) is weakly strongly compact. Consequently, by Propositions 3.3 and 3.5, \(\mathord{\uparrow} f(A_i \cap \overline{A})\) is weakly strongly compact as well.

Claim 2: $\overline{f(A)} \in m(\mathcal{G})$.

For each $A_i \in \mathcal{F}$, since $\overline{A} \cap A_i \neq \emptyset$, there exists some $a \in \overline{A} \cap A_i$. Thus $f(a) \in f(\overline{A}) \subseteq \overline{f(A)}$. Hence,
\[
f(a) \in \overline{f(A)} \cap f(A_i \cap \overline{A}),
\]
It follows that $\overline{f(A)} \in M(\mathcal{G})$.

Now suppose $C \subseteq \overline{f(A)}$ and $C \in M(\mathcal{G})$. Then for every $A_i \in \mathcal{F}$,
\[
C \cap f(A_i \cap \overline{A}) \neq \emptyset.
\]
It follows that
\[
f^{-1}(C) \cap \overline{A} \cap A_i \neq \emptyset
\]
for all $A_i \in \mathcal{F}$. Hence, $f^{-1}(C) \cap \overline{A} \in M(\mathcal{F})$. By the minimality of $\overline{A}$ in $M(\mathcal{F})$, we obtain
\[
\overline{A} = f^{-1}(C) \cap \overline{A}.
\]
Thus $\overline{A} \subseteq f^{-1}(C)$, which implies $f(\overline{A}) \subseteq C$, so $\overline{f(A)} \subseteq C$. Therefore, $C = \overline{f(A)}$; consequently, $\overline{f(A)} \in m(\mathcal{G})$.

Obviously, $\mathcal{G}$ is filtered. We conclude taht $f(A)$ is a $\mathcal{U}_c$-set in $Y$.
   
\end{proof}

In the following theorem, we give some equivalent characterizations of \(\mathcal{U}_c\)-admitting spaces.
\begin{theorem}
    For every $T_0$ space $X$, the following statements are equivalent:
    \begin{enumerate}
    \item \( X \) is \( \mathcal{U}_c \)-admitting.
    \item $\mathcal{F} \subseteq \mathcal{Q}_c(X)$ a filtered family, and $U \in \mathcal{O}(X)$. If $\bigcap \mathcal{F} \subseteq U$, then there exists some $K \in \mathcal{F}$ such that $K \subseteq U$.
    \item $\mathcal{F} \subseteq \mathcal{Q}_c(X)$ a filtered family, and $A$ is closed. If $A \cap K \neq \emptyset$ for every $K \in \mathcal{F}$, then, $A \cap \bigcap \mathcal{F} \neq \emptyset$.
    \item For each \( \mathcal{U}_c \)-set \( A \subseteq X \) and \( U \in \mathcal{O}(X) \), \( \bigcap_{a \in A} \mathord{\uparrow} a \subseteq U \) implies \( \mathord{\uparrow} a \subseteq U \) for some \( a \in A \).
    \item For each closed \( \mathcal{U}_c \)-set \( A \subseteq X \) and \( U \in \mathcal{O}(X) \), \( \bigcap_{a \in A} \mathord{\uparrow} a \subseteq U \) implies \( \mathord{\uparrow} a \subseteq U \) for some \( a \in A \).
\end{enumerate}
\end{theorem}
\begin{proof}
    2 $\Rightarrow$ 1: Suppose $A$ is a $\mathcal{U}_c$-set in $X$. Then there exists a filtered family $\mathcal{F} = \{K_i : i \in I\} \subseteq \mathcal{Q}_c(X)$ such that $\overline{A} \in m(\mathcal{F})$. We claim that $\overline{A} \cap \bigcap \mathcal{F} \neq \emptyset$. Suppose not,  $\bigcap \mathcal{F} \subseteq X \setminus \overline{A}$. It follows that there exists some $K_i \in \mathcal{F}$ such that $K_i \subseteq X \setminus \overline{A}$, which contradicts $K_i \cap \overline{A} \neq \emptyset$. Hence, $\overline{A} \cap \bigcap \mathcal{F} \neq \emptyset$. Take a point \( x \in \overline{A} \cap \bigcap \mathcal{F} \). Since \( x \in \overline{A} \), we have \( \mathord{\downarrow} x \subseteq \overline{A} \). Because \( x \in K_i \) for every \( i \in I \), it follows that \( \mathord{\downarrow} x \cap K_i \neq \emptyset \) for all \( i \). So \( \mathord{\downarrow} x \in M(\mathcal{F}) \). By the minimality of \( \overline{A} \), we conclude that \( \overline{A} = \mathord{\downarrow} x \).
    
    1 $\Rightarrow$ 5: Since $A = \mathord{\downarrow} x$, we have $\bigcap_{a \in A} \mathord{\uparrow} a = \bigcap_{a \in \mathord{\downarrow} x} \mathord{\uparrow} a = \mathord{\uparrow} x \subseteq U$.

5 $\Rightarrow$ 4: Suppose that $\mathord{\uparrow} a \nsubseteq U$ for every $a \in A$. Then for each $a \in A$, we have $\mathord{\uparrow} a \cap (X \setminus U) \neq \emptyset$, which is equivalent to $a \in X \setminus U$. Hence, $A \subseteq X \setminus U$. Since $X \setminus U$ is closed, it follows that $\overline{A} \subseteq X \setminus U$. Therefore,
\[
\bigcap_{a \in \overline{A}} \mathord{\uparrow} a \subseteq U.
\]
Thus, there exists some $a \in \overline{A}$ such that $\mathord{\uparrow} a \subseteq U$. In particular, $a \in U$, so $\overline{A} \cap U \neq \emptyset$, which contradicts $\overline{A} \subseteq X \setminus U$. So there must exist some $a \in A$ such that $\mathord{\uparrow} a \subseteq U$.

4 $\Rightarrow$ 2: Suppose that for every $K \in \mathcal{F}$, we have $K \cap (X \setminus U) \neq \emptyset$. By topological Rudin lemma, there exists a minimal closed irreducible set $A \subseteq X \setminus U$ such that $A \cap K \neq \emptyset$ for all $K \in \mathcal{F}$. Since $A$ is closed and intersects every member of $\mathcal{F}$, it follows that $A$ is a $\mathcal{U}_c$-set and $\bigcap_{a \in A} \mathord{\uparrow} a \subseteq \bigcap \mathcal{F} \subseteq U$. Thus there exists some $a \in A$ such that $\mathord{\uparrow} a \subseteq U$. This implies $a \in U$, contradicting $A \subseteq X \setminus U$. Therefore, there exists some $K \in \mathcal{F}$ such that $K \subseteq U$.

2 $\Leftrightarrow$ 3: Obviously.
\end{proof}

\begin{proposition}
    For a countable $T_0$ space, it is well-filtered if and only if it is \(\mathcal{U}_c\)-admitting.
\end{proposition}
\begin{proof}
Let \(X\) be a countable \(T_0\) space. suppose \(X\) is \(\mathcal{U}_c\)-admitting. Since \(X\) is countable, every compact saturated subset of \(X\) is weakly strongly compact. Thus \(\mathcal{Q}(X) = \mathcal{Q}_c(X)\);  consequently, every filtered family of compact saturated sets is a filtered family in \(\mathcal{Q}_c(X)\). 

Now let \(\mathcal{F} \subseteq \mathcal{Q}(X)\) be a filtered family and let \(U \in \mathcal{O}(X)\) with \(\bigcap \mathcal{F} \subseteq U\). Since \(X\) is \(\mathcal{U}_c\)-admitting, by Theorem 3.13, there exists some \(K \in \mathcal{F}\) such that \(K \subseteq U\). Therefore, \(X\) is well-filtered.
\end{proof}

The following counterexample shows that our newly introduced \(\mathcal{U}_c\)-admitting spaces are meaningful.
\begin{example}\label{ex2}
    Consider \((X, \tau_0)\), where \(\omega_1\) is the first uncountable ordinal, \(X = [0, \omega_1)\) be the set of all ordinals less than \(\omega_1\)  and \(\tau_0\) is its order topology. That is, the topology is generated by the base
\[
\mathcal{B} = \{ (a, b) : a, b \in X \} \cup \{ [0, a) : a \in X \} \cup \{ (a, \omega_1) : a \in X \}.
\]
    Let \((X, \tau)\) the co-compact topology on \((X, \tau_0)\), where
$\tau = \{\emptyset\} \cup \{ X \setminus C : C \text{ is compact in } (X, \tau_0) \}$.

We verify that \((X, \tau)\) is \(\mathcal{U}_c\)-admitting but not well-filtered.

   Claim 1: In \((X, \tau_0)\), compactness is equivalent to being closed and bounded.

Let \(K \subseteq X\) be compact. Obviously, \((X, \tau_0)\) is $T_2$, so \(K\) is closed. Assume \(K\) is unbounded. Then the open cover \(\{[0, a) : a \in X\}\) of \(K\) has no finite subcover, a contradiction. Hence, \(K\) is bounded. 

Let \(K \subseteq X\) be closed and bounded. Then there exists \(a \in X\) such that \(K \subseteq [0, a]\). We prove by transfinite induction that for every \(b \le a\), the set \([0, b]\) is compact.

When \(b = 0\), the result is immediate.

Assume that for a fixed \(b\), \([0, b]\) is compact. For the successor case \(b+1\), let \(\mathcal{V}\) be an open cover of \([0, b+1]\). So \([0, b]\) has a finite subcover, and adding an open set from \(\mathcal{V}\) that contains \(b+1\), which yields a finite subcover of \([0, b+1]\). Consequently, \([0, b+1]\) is compact.

Now suppose \(b\) is a limit ordinal in \(X\) and assume that \([0, c]\) is compact for all \(c < b\). Let \(\mathcal{V}\) be an open cover of \([0, b]\). Then there exists some \(V \in \mathcal{V}\) containing \(b\). Since \(V\) is open, there is \(c \in [0, b)\) such that \([c, b] \subseteq V\). By the induction hypothesis, \([0, c]\) has a finite subcover. Together with \(V\), this gives a finite subcover of \([0, b]\). Thus \([0, b]\) is compact.

Therefore, by transfinite induction, \([0, a]\) is compact. Since \(K\) is a closed subset of \([0, a]\), it follows that $K$ is compact in \(X\).

Cliam 2: $\tau$ is a topology.

Obviously, \(\emptyset \in \tau\) and \(X = X \setminus \emptyset \in \tau\). Suppose \(X \setminus C_1, X \setminus C_2 \in \tau\), then
$(X \setminus C_1) \cap (X \setminus C_2) = X \setminus (C_1 \cup C_2) \in \tau$ because the finite unions of compact sets are compact.
Suppose \(X \setminus C_i \in \tau\) for \(i \in I\), then
$\bigcup_{i \in I} (X \setminus C_i) = X \setminus \bigcap_{i \in I} C_i \in \tau$,
due to the fact that in \((X, \tau_0)\), compact sets are exactly closed and bounded sets, and arbitrary intersections of closed and bounded sets are again closed and bounded; hence, compact.

Claim 3:  \(\mathcal{Q}_c(X, \tau) = \mathcal{Q}(X, \tau_0)\).

Since every singleton \(\{a\}\) is compact in \((X, \tau_0)\), the set \(X \setminus \{a\}\) is open in \((X, \tau)\). Thus \(\{a\}\) is closed in \((X, \tau)\). It shows that \((X, \tau)\) is \(T_1\), therefore, the weakly strongly compact saturated sets in \(\tau\) are precisely the countable and compact sets.

For any \( U \in \tau \), the set \( X \setminus U \) is compact in \((X, \tau_0)\). Since \((X, \tau_0)\) is $T_2$, \( X \setminus U \) is closed in \((X, \tau_0)\), that is, \( U \in \tau_0 \). This shows $\tau \subseteq \tau_0$.

Let \(K\) be a countable and compact set in \(\tau\). Then \(\sup K\) exists and is less than \(\omega_1\) and \(K\) is bounded. Suppose \(K\) is not closed in \((X, \tau_0)\), then there exists \(p \in \operatorname{cl}_{\tau_0}(K) \setminus K\) with \(p \le \sup K\). Since \(p\) is a countable ordinal, there exists a countable sequence \(\{x_n\}_{n \in X} \subseteq K\) converging to \(p\) in \(\tau_0\). For each \(n \in \mathbb{N}\), let \(C_n = \{p\} \cup \{x_m : m \ge n\}\). As \(p\) is a limit ordinal, each \(C_n\) is compact in \((X, \tau_0)\). Set \(V_n = X \setminus C_n \in \tau\). Now let \(y \in K\). If \(y = x_m\) for some \(m \in \mathbb{N}\), then \(y \notin C_{m+1}\), so \(y \in V_{m+1}\); if \(y \notin \{x_n : n \in \mathbb{N}\}\), then \(y \notin C_0\). In either case, \(y \in V_0\), hence,
$K \subseteq \bigcup_{n \in \mathbb{N}} V_n$.
For any finite subfamily of \(\{V_n : n \in \mathbb{N}\}\), let \(n\) be the largest index among them. Then the union of these finitely many \(V_n\)'s is just \(V_n\), but \(x_n \notin V_n\), it means that \(\{V_n : n \in \mathbb{N}\}\) has no finite subcover of \(K\), contradicting the compactness of \(K\) in \((X, \tau)\). So \(K\) is closed in \((X, \tau_0)\); therefore, \(K\) is compact in \((X, \tau_0)\).

Let \(K\) be a compact set in \((X, \tau_0)\). Since \(\tau \subseteq \tau_0\), \(K\) is also compact in \((X, \tau)\). Moreover, \(K\) is closed and bounded in \((X, \tau_0)\), so it is countable. Hence, \(K\) is a countable and compact set in \((X, \tau)\).

 Thus, \(\mathcal{Q}_c(X, \tau) = \mathcal{Q}(X, \tau_0)\).

Claim 4: \((X, \tau)\) is \(\mathcal{U}_c\)-admitting.

 Suppose \(\mathcal{F} = \{K_i:i \in I\} \subseteq \mathcal{Q}_c(X, \tau)\) is a filtered family and \(U \in \tau\) such that \(\bigcap \mathcal{F} \subseteq U\). If \(U = X\), then clearly, \(K_i \subseteq U\) for any \(i \in I\). Now assume \(U \neq X\). Since \(\tau \subseteq \tau_0\), we have \(U \in \tau_0\). By Claim 3, \(\mathcal{F}\) is also a filtered family in \(\mathcal{Q}(X, \tau_0)\). As \((X, \tau_0)\) is $T_2$, it is well-filtered. So there exists some \(K_i \in \mathcal{F}\) such that \(K_i \subseteq U\). It follows that \((X, \tau)\) is \(\mathcal{U}_c\)-admitting.

Claim 5: \((X, \tau)\) is not well-filtered.

We first prove that every closed set in \((X, \tau_0)\) is compact in \((X, \tau)\). Let \(K\) be closed in \((X, \tau_0)\) and let \(\{U_i : i \in I\} \subseteq \tau\) cover \(K\). If \(K = \emptyset\), it is trivially compact. Otherwise, pick \(U_{i_0} = X \setminus C_{i_0} \in \tau\). Then \(K \setminus U_{i_0} = K \cap C_{i_0}\) is compact in \((X, \tau_0)\), so it has a finite subcover. Adding \(U_{i_0}\) yields a finite subcover of \(K\). Thus \(K\) is compact in \((X, \tau)\).

For each \(a \in X\), the set \([a, \omega_1)\) is closed in \((X, \tau_0)\); hence it is compact in \((X, \tau)\). The family
$\{ [a, \omega_1) : a \in X \}$
is filtered, but its intersection
$\bigcap_{a \in X} [a, \omega_1) = \emptyset$.
Therefore, \((X, \tau)\) is not well-filtered.

\end{example}

We now turn to the relationship between \(\mathcal{U}_c\)-admitting spaces and sober spaces.
\begin{remark}
    For a \(T_0\) space, \(\mathcal{U}_c\)-admitting property together with local compactness, or with core compactness, does not imply sobriety.
\end{remark}
In Example 3.15, the space \((X, \tau)\) is \(\mathcal{U}_c\)-admitting but not well-filtered, hence not sober. We now show that it is locally compact; as a consequent, it is also core compact.
\begin{proof}
    Cliam 1: In \((X, \tau)\), \((a, \omega_1)\) is always compact for every \(a \in X\).

    Suppose \((a, \omega_1) \subseteq \bigcup_{i \in I} U_i\), where each \(U_i = X \setminus C_i \in \tau\) and \(C_i\) is compact in \(\tau_0\). Choose some \(U_{i_0}\) with \(i_0 \in I\). Then
$(a, \omega_1) \setminus U_{i_0} = (a, \omega_1) \cap C_{i_0}$.
Since \((a, \omega_1) = X \setminus [0, a+1)\), it is closed in \(\tau_0\). Thus \((a, \omega_1) \cap C_{i_0}\) is compact in \((X, \tau_0)\). Because \(\tau \subseteq \tau_0\), this set is also compact in \((X, \tau)\). So it has a finite subcover from \(\{U_i : i \in I\}\). Together with \(U_{i_0}\), this yields a finite subcover of \((a, \omega_1)\). Therefore, \((a, \omega_1)\) is compact in \((X, \tau)\).

Claim 2: In \((X, \tau)\), \((a, \omega_1)\) is open for every \(a \in X\).

By the proof of Example 3.15, \([0, a]\) is compact in \((X, \tau)\). Hence, \(X \setminus [0, a]\) is open in \(\tau\). Since
$(a, \omega_1) = X \setminus [0, a]$,
it follows that \((a, \omega_1) \in \tau\).

Claim 3: \((X, \tau)\) is locally compact.

Let \(a \in X\) and \(U \in \tau\) with \(a \in U\). Since \(X \setminus U\) is compact in \((X, \tau_0)\), it is bounded. Thus there is \(b \in X\) such that \((b, \omega_1) \subseteq U\).

If \(a\) is not a limit ordinal, set \(K = (a-1, a+1) \cup (b, \omega_1)\). By Claim 1, \((b, \omega_1)\) is compact, and \((a-1, a+1)\) is finite, so \(K\) is compact in \(\tau\). Also \(X \setminus K = [0, a] \cup [a+1, b]\) and is compact in \(\tau_0\), hence, \(K \in \tau\). Thus \(a \in \operatorname{int}K \subseteq K \subseteq U\).

If \(a\) is a limit ordinal, since \(X \setminus U\) is closed in \(\tau_0\), there exists \(c \in X\) such that \((c, a+1) \subseteq U\); otherwise, a sequence from \(X \setminus U\) would converge to \(a\), forcing \(a \in X \setminus U\). Then the set \((c, a+1)\) is compact in \(\tau_0\), hence, it is compact in \(\tau\). Let \(K = (c, a+1) \cup (b, \omega_1)\). Then \(K\) is compact, open, and satisfies \(a \in \operatorname{int}K \subseteq K \subseteq U\). So \((X, \tau)\) is locally compact.
\end{proof}
In analogy with the well-filtered setting, we shall examine under what additional  conditions, a \(\mathcal{U}_c\)-admitting space is sober. Below we provide two sufficient conditions.
\begin{definition}
    A topological space \(X\) is called \emph{locally countably compact} if for every point \(x \in X\) and every open neighbourhood \(U\) of \(x\), there exists a countable set \(C \subseteq X\) such that
$x$ $\in \operatorname{int}(\mathord{\uparrow} C) \subseteq \mathord{\uparrow} C \subseteq U$.
\end{definition}

\begin{remark}
    Every locally finitely compact space is locally countably compact, and every locally countably compact space is locally compact.
\end{remark}

\begin{theorem}
    Every locally countably compact \(\mathcal{U}_c\)-admitting space is sober.
\end{theorem}
\begin{proof}
    Let $A$ be a nonempty irreducible closed subset of $X$. Define
\[
\mathcal{F} = \{ \mathord{\uparrow} F : F \subseteq X \text{ is countable and } \operatorname{int}(\mathord{\uparrow} F) \cap A \neq \emptyset \}.
\]

We first show that $\mathcal{F}$ is filtered. Since $A \neq \emptyset$, choose $x \in A$. As $X$ is locally countably compact, there exists a countable set $F \subseteq X$ such that $x \in \operatorname{int}(\mathord{\uparrow} F) \subseteq \mathord{\uparrow} F$.
Hence, $\operatorname{int}(\mathord{\uparrow} F) \cap A \neq \emptyset$, so $\mathcal{F} \neq \emptyset$.
Now take $\mathord{\uparrow} F_1, \mathord{\uparrow} F_2 \in \mathcal{F}$. Then
\[
\operatorname{int}(\mathord{\uparrow} F_1) \cap A \neq \emptyset
\quad\text{and}\quad
\operatorname{int}(\mathord{\uparrow} F_2) \cap A \neq \emptyset.
\]
Since $A$ is irreducible, it follows that
\[
A \cap \operatorname{int}(\mathord{\uparrow} F_1) \cap \operatorname{int}(\mathord{\uparrow} F_2) \neq \emptyset.
\]
Pick $x \in A \cap \operatorname{int}(\mathord{\uparrow} F_1) \cap \operatorname{int}(\mathord{\uparrow} F_2)$. By local countable compactness, there exists a countable set $F_3 \subseteq X$ such that
\[
x \in \operatorname{int}(\mathord{\uparrow} F_3) \subseteq \mathord{\uparrow} F_3 \subseteq \operatorname{int}(\mathord{\uparrow} F_1) \cap \operatorname{int}(\mathord{\uparrow} F_2).
\]
Thus $\mathord{\uparrow} F_3 \in \mathcal{F}$ and $\mathord{\uparrow} F_3 \subseteq \mathord{\uparrow} F_1 \cap \mathord{\uparrow} F_2$. Therefore $\mathcal{F}$ is filtered.

By topological Rudin lemma, there exists a minimal closed subset $C \subseteq A$ such that $C \cap K \neq \emptyset$ for every $K \in \mathcal{F}$. Since $C$ is a closed set intersecting every member of $\mathcal{F}$, it is a $\mathcal{U}_c$-set. Because $X$ is $\mathcal{U}_c$-admitting, we have
$C \cap \bigcap \mathcal{F} \neq \emptyset$.
In particular, $A \cap \bigcap \mathcal{F} \neq \emptyset$. Choose $x_0 \in A \cap \bigcap \mathcal{F}$. Then $x_0 \in A$, so $\mathord{\downarrow} x_0 \subseteq A$.
Now let $y \in A$ and let $U$ be any open neighbourhood of $y$. By local countable compactness, there exists a countable set $F_y \subseteq X$ such that
\[
y \in \operatorname{int}(\mathord{\uparrow} F_y) \subseteq \mathord{\uparrow} F_y \subseteq U.
\]
Then $\mathord{\uparrow} F_y \in \mathcal{F}$, so $x_0 \in \mathord{\uparrow} F_y \subseteq{U}$. Consequently, $y \le x_0$ in the specialization order, i.e., $y \in \mathord{\downarrow} x_0$. Thus $A \subseteq \mathord{\downarrow} x_0$. Combining with $\mathord{\downarrow} x_0 \subseteq A$, we get $A = \mathord{\downarrow} x_0$. Therefore, every irreducible closed set is the closure of a singleton. So $X$ is sober.
\end{proof}

\begin{proposition}
    Every $A_1$  \(\mathcal{U}_c\)-admitting space is sober.
\end{proposition}
The proof of this proposition follows almost the same steps as that of the result in paper \cite[Theorem 4.2]{XuShen2020} that every $A_1$ well-filtered $T_0$ space is sober. However, for the convenience of readers, we include the full proof here.
\begin{proof}
    For each \( x \in X \), since \( X \) is first countable, there is an open neighborhood base \(\{U_n(x) : n \in \mathbb{N}\}\) of \( x \) such that
\[
U_1(x) \supseteq U_2(x) \supseteq \cdots \supseteq U_k(x) \supseteq \cdots,
\]
that is, \(\{U_n(x) : n \in \mathbb{N}\}\) is a decreasing sequence of open subsets.

Suppose that \( A \in \operatorname{IRR}(X) \). We show that \( \overline{A} \) is directed. Let \( a_1, a_2 \in \overline{A} \). It suffices to show that \(\mathord{\uparrow} a_1 \cap \mathord{\uparrow} a_2 \cap \overline{A} \neq \emptyset \). Since \( a_1, a_2 \in \overline{A} \),
\[
A \cap U_1(a_1) \neq \emptyset \neq A \cap U_1(a_2),
\]
and hence, \( A \cap U_1(a_1) \cap U_1(a_2) \neq \emptyset \) by the irreducibility of \( A \). Choose \( c_1 \in U_1(a_1) \cap U_1(a_2) \cap A \). Now suppose we already have a set \( \{c_1, c_2, \ldots, c_n\} \) such that for each \( 2 \leq i \leq n \),
\[
c_i \in U_i(c_1) \cap U_i(c_2) \cap \cdots \cap U_i(c_i-1) \cap U_i(a_1) \cap U_i(a_2) \cap A.
\]
Note that the above condition implies that for any positive integer \( k \),
\[
U_k(c_1) \cap U_k(c_2) \cap \cdots \cap U_k(c_n) \cap U_k(a_1) \cap U_k(a_2) \cap A \neq \emptyset.
\]
So we can choose
\[
c_{n+1} \in U_{n+1}(c_1) \cap U_{n+1}(c_2) \cap \cdots \cap U_{n+1}(c_n) \cap U_{n+1}(a_1) \cap U_{n+1}(a_2) \cap A.
\]
By induction, we can obtain a set \(\{c_n : n \in \mathbb{N}\}\).

Let \(K_n = \mathord{\uparrow} \{c_k : k \geq n\}\) for each \(n \in \mathbb{N}\).

Claim 1: \(\forall n \in \mathbb{N}, K_n\) is locally weakly strongly compact.

Suppose \(\{V_i : i \in I\}\) is an open cover of \(K_n\), i.e., \(K_n \subseteq \bigcup_{i \in I} V_i\). Then there is \(i_0 \in I\) such that \(c_n \in V_{i_0}\); thus there is \(m \geq n\) such that  
\(c_n \in U_m(c_n) \subseteq V_{i_0}\). It follows that \(c_k \in U_m(c_n) \subseteq V_{i_0}\) for all \(k \geq m\). Thus  
\(\{c_k : k \geq m\} \subseteq V_{i_0}\). For each \(c_k\), where \(n \leq k \leq m\), choose a \(V_{i_k}\) such that  
\(c_k \in V_{i_k}\). Then the finite family \(\{V_{i_k} : n \leq k \leq m\} \cup \{V_{i_0}\}\) covers \(K_n\). So \(K_n\) is compact, $\{c_k : k \geq n\}$ is countable, hence, it is locally weakly strongly compact.

Claim 2: \(\bigcap_{n \in \mathbb{N}} K_n \cap \overline{A} \neq \emptyset\).

Clearly, \(\{K_n : n \in \mathbb{N}\}\) is filtered, and \(\overline{A} \cap K_n \neq \emptyset\) for each \(n \in \mathbb{N}\). Since \(X\) is a \(\mathcal{U}_c\)-admitting space, it follows that
$\bigcap_{n \in \mathbb{N}} K_n \cap \overline{A} \neq \emptyset$.

Claim 3: \(\mathord{\uparrow} a_1 \cap \mathord{\uparrow} a_2 \cap \overline{A} \neq \emptyset\).

Note that for each \( n \in \mathbb{N} \), \(K_n = \{c_k : k \geq n\} \subseteq U_n(a_1) \cap U_n(a_2)\). This implies that
\[
\emptyset \neq \bigcap_{n \in \mathbb{N}} K_n \cap \overline{A}
\subseteq \bigcap_{n \in \mathbb{N}} U_n(a_1) \cap \bigcap_{n \in \mathbb{N}} U_n(a_2) \cap \overline{A}
= \mathord{\uparrow} a_1 \cap \mathord{\uparrow} a_2 \cap \overline{A}.
\]
Thus \(\mathord{\uparrow} a_1 \cap \mathord{\uparrow} a_2 \cap \overline{A} \neq \emptyset\),  there is \(a_3 \in \mathord{\uparrow} a_1 \cap \mathord{\uparrow} a_2 \cap \overline{A}\), that is, \(a_3 \in \overline{A}\) such that \(a_1 \leq a_3\) and \(a_2 \leq a_3\). Therefore, \(\overline{A}\) is directed. Since $X$ is \(\mathcal{U}_c\)-admitting, thus a d-space. As a result, \(\overline{A}\) is the closure of a singleton. Therefore, $X$ is sober. 
\end{proof}

\section{H-sober approach to \(\mathcal{U}_c\)-admitting spaces}
In \cite{XuHsober2021}, Xu introduced the notions of subset systems and H-sober spaces, and employed them to study in a unified manner for various properties of \(d\)-spaces, well-filtered spaces, and sober spaces. In this section, we shall also adopt the H-sober perspective to investigate the related properties of \(\mathcal{U}_c\)-admitting spaces, and establish their connections with the corresponding power spaces.

\begin{definition}(\cite[Definition 3.1]{XuHsober2021})
    A covariant functor \( H : \mathbf{Top}_0 \to \text{Set} \) is called a \emph{subset system} on \( \mathbf{Top}_0 \) (the category of \(T_0\) spaces and continuous maps) provided that the following two conditions are satisfied:
\begin{enumerate}
    \item \( S(X) \subseteq H(X) \subseteq 2^X \) (the set of all subsets of \( X \)) for each \( X \in \operatorname{ob}(\mathbf{Top}_0) \).
    \item For any continuous mapping \( f : X \to Y \) in \( \mathbf{Top}_0 \), \( H(f)(A) = f(A) \in H(Y) \) for all \( A \in H(X) \).
\end{enumerate}

A subset system \( H : \mathbf{Top}_0 \to \text{Set} \) is called an \emph{irreducible subset system}, or an \emph{\(R\)-subset system} for short, if \( H(X) \subseteq \operatorname{Irr}(X) \) for all \( X \in \operatorname{ob}(\mathbf{Top}_0) \). 
\end{definition}

\begin{definition}(\cite[Definition 4.1]{XuHsober2021})
    Let \( H : \mathbf{Top}_0 \to \text{Set} \) be an R-subset system. A \( T_0 \) space \( X \) is called \( H \)-sober if for any \( A \in H(X) \), there is a (unique) point \( x \in X \) such that \( \overline{A} = \overline{\{x\}} \) or, equivalently, if \( H_c(X) = \mathcal{S}_c(X) \). The category of all H-sober spaces with continuous mappings is denoted by \textbf{H-Sob}.
\end{definition}

\begin{example}
    \begin{enumerate}
    \item $\mathcal{D}$-spaces are exactly $\mathcal{D}$-sober spaces ($\mathcal{D}(X) = \{ A \subseteq X : A \text{ is directed} \}$).
    \item Well-filtered spaces are exactly $\mathcal{KF}$-sober spaces ($\mathcal{KF}(X) = \{ A \subseteq X : A \text{ is a KF-set} \}$).
    \item Sober spaces are exactly $\mathcal{R}$-sober spaces ($\mathcal{R}(X) = \{ A \subseteq X : A \text{ is irreducible} \}$).
    \end{enumerate}
\end{example}

\begin{proposition}
    $\mathcal{U}_s$ and $\mathcal{U}_c$ are $\mathcal{R}$-subset systems, $\mathcal{U}_s$-admitting spaces are exactly $\mathcal{U}_s$-sober spaces;  $\mathcal{U}_c$-admitting spaces are exactly $\mathcal{U}_c$-sober spaces.
\end{proposition}
\begin{proof}
    By Proposition 3.12, the continuous image of a $\mathcal{U}_c$-set is a $\mathcal{U}_c$-set. Thus all of them satisfy the two conditions in Definition 4.1. Moreover, since every \(\mathcal{U}_s\)-set and every \(\mathcal{U}_c\)-set is irreducible, they are all \(R\)-subset systems. Obviously, $\mathcal{U}_s$-admitting spaces are exactly $\mathcal{U}_s$-sober spaces;  $\mathcal{U}_c$-admitting spaces are exactly $\mathcal{U}_c$-sober spaces.
\end{proof}

From \cite{XuHsober2021}, we know many properties of H-sober spaces. Here we list some important ones.
\begin{lemma}
    Let \( H : \mathbf{Top}_0 \to \text{Set} \) be an subset system and \( X \) be an \( H \)-sober space, the following facts hold.
    \begin{enumerate}
    \item If \( A \) is a nonempty closed subspace of \( X \), then \( A \) is \( H \)-sober. If \( U \) is a nonempty saturated subspace of \( X \), then \( U \) is \( H \)-sober.
    \item A retract of an \( H \)-sober space is \( H \)-sober.
    \item Let \(\{X_i : i \in I\}\) be a family of \(T_0\) spaces. The product space \(\prod_{i \in I} X_i\) is \(H\)-sober if and only if for each \(i \in I\), \(X_i\) is \(H\)-sober.
    \item For any \( R \)-subset system \( H : \mathbf{Top}_0 \to \text{Set} \), \( H \)-Sob is complete.
    \item Let \( H : \mathbf{Top}_0 \longrightarrow \text{Set} \) be an \( R \)-subset system. Then \textbf{H-Sob} is reflective in \(\mathbf{Top}_0\).
    \end{enumerate}
\end{lemma}

By Propositions 4.4 and Lemma 4.5, we have the following conclusions.
\begin{corollary}
\mbox{}
   \begin{enumerate}
       \item Nonempty closed subspaces and nonempty saturated subspaces of a \(\mathcal{U}_c\)-admitting space are again \(\mathcal{U}_c\)-admitting.
        \item Retracts of a \(\mathcal{U}_c\)-admitting space are again \(\mathcal{U}_c\)-admitting.
        \item The product of a family of topological spaces is \(\mathcal{U}_c\)-admitting if and only if each factor is \(\mathcal{U}_c\)-admitting.
        \item \( \mathcal{U}_c \)-Sob is complete.
        \item \( \mathcal{U}_c \)-Sob is reflective in \(\mathbf{Top}_0\).
    \end{enumerate}
\end{corollary}

In \cite{XuHsober2021}, the relationship between \(\mathcal{U}_s\)-admitting spaces and their power spaces was not investigated. However, in \cite{Shen2025}, the relationship between \(\mathcal{U}_s\)-admitting spaces and their power spaces was given.

For a topological space \( X \), the \emph{upper topology} on \( \mathcal{Q}_s^*(X) = \mathcal{Q}_s(X) \setminus \{\emptyset\} \) is generated by the family as a base
\[
\Box U = \{ K \in \mathcal{Q}_s^*(X) : K \subseteq U \},
\]
where \( U \) ranges over the open subsets of \( X \). The resulting space, denoted by \( P_\text{us}(X) \), is called the \emph{$\mathcal{U}_s$-Smyth power space} of $X$.

\begin{proposition}(\cite[Theorem 4.3]{Shen2025})
    Let \( X \) be a \( T_0 \) space. Then the following conditions are equivalent:
\begin{enumerate}
    \item \( P_\text{us}(X) \) is \( \mathcal{U}_s \)-admitting.
    \item \( P_\text{us}(X) \) is a \( d \)-space.
    \item \( X \) is \( \mathcal{U}_s \)-admitting.
\end{enumerate}
\end{proposition}

Similarly, we shall prove the connections between  \(\mathcal{U}_c\)-admitting spaces and their power spaces.

For a topological space \( X \), the \emph{upper topology} on \( \mathcal{Q}_c^*(X) = \mathcal{Q}_c(X) \setminus \{\emptyset\} \) is generated by the family as a base
\[
\Box U = \{ K \in \mathcal{Q}_c^*(X) : K \subseteq U \},
\]
where \( U \) ranges over the open subsets of \( X \). The resulting space, denoted by \( P_\text{uc}(X) \), is called the \emph{$\mathcal{U}_c$-Smyth power space} of $X$.

\begin{theorem}
     Let \( X \) be a \( T_0 \) space. Then the following conditions are equivalent:
\begin{enumerate}
    \item \( P_\text{uc}(X) \) is \( \mathcal{U}_c \)-admitting.
    \item \( P_\text{uc}(X) \) is a \( d \)-space.
    \item \( X \) is \( \mathcal{U}_c \)-admitting.
\end{enumerate}
\end{theorem}
\begin{proof}
    1 $\Rightarrow$ 2: Obviously.

    2 $\Rightarrow$ 3: Suppose \(\mathcal{F} = \{K_i : i \in I\} \subseteq \mathcal{Q}_c(X)\) is a filtered family, \(U \in \mathcal{O}(X)\), and \(\bigcap \mathcal{F} \subseteq U\). Since  \( P_\text{uc}(X) \) is a \(d\)-space, the supremum of \(\mathcal{F}\) exists in  \( P_\text{uc}(X) \); we denote it by \(K_0\). 

Now, \(\square U\) is an open set in  \( P_\text{uc}(X) \), and as \(K_0 \subseteq \bigcap \mathcal{F} \subseteq U\), we have \(K_0 \in \square U\). Because open sets in a \(d\)-space are Scott open, and \(\mathcal{F}\) is directed with \(\sup \mathcal{F} = K_0 \in \square U\), there exists some \(i \in I\) such that \(K_i \in \square U\), i.e., \(K_i \subseteq U\). Hence, \(X\) is \(\mathcal{U}_c\)-admitting.

 3 $\Rightarrow$ 1: Suppose \(\mathcal{F} = \{\mathcal{A}_i : i \in I\}\) is a filtered family in \(\mathcal{Q}_c( P_\text{uc}(X) )\), \(U \in \mathcal{O}(X)\), and
$\bigcap \mathcal{F} \subseteq \square U$.
Assume that no \(\mathcal{A}_i\) is contained in \(\square U\). That is, for every \(i \in I\),
$\mathcal{A}_i \cap ( P_\text{uc}(X)  \setminus \square U) \neq \emptyset$.
By Proposition 3.6, $\mathcal{A}_i \cap ( P_\text{uc}(X)  \setminus \square U)$ is weakly strongly compact. Let
\[
K_i = \bigcup \big( \mathcal{A}_i \cap (P_\text{uc}(X) \setminus \square U) \big).
\]

Claim:  \(K_i\) is weakly strongly compact in $X$.

Suppose \( V \) is an open neighbourhood of \( K_i \). Then
\[
\mathcal{A}_i \cap \big( P_\text{uc}(X) \setminus \square U \big) \subseteq \square V.
\]
Thus there exists a countable family \( \mathcal{B}_i = \{ Q_n : n \in N_i \} \subseteq \mathcal{Q}_c(X) \) such that
\[
\mathcal{A}_i \cap \big( P_\text{uc}(X) \setminus \square U \big) \subseteq \mathord{\uparrow} \mathcal{B}_i \subseteq \square V.
\]
Since each \( Q_n \) is weakly strongly compact in $X$, there exists a countable set \( M_n \subseteq X \) such that
\[
Q_n \subseteq \mathord{\uparrow} M_n \subseteq V.
\]
Let \( F = \bigcup_{n \in N_i} M_n \). Then \( F \) is countable by the fact that each \( M_n \) is countable. Now for any \( K_\text{in} \subseteq \mathcal{A}_i \cap \big( P_\text{uc}(X) \setminus \square U \big) \), there exists \( Q_n \) such that
\[
K_\text{in} \subseteq \mathord{\uparrow} Q_n \subseteq \mathord{\uparrow} M_n \subseteq \mathord{\uparrow} F \subseteq V.
\]
So $K_i \subseteq \mathord{\uparrow} F \subseteq V$. Thus \( K_\text{i} \) is weakly strongly compact.

Clearly, each \(K_i\) is saturated and \(K_i \nsubseteq U\); equivalently, \(K_i \cap (X \setminus U) \neq \emptyset\) for every \(i \in I\). By topological Rudin lemma, there exists an irreducible closed set \(A \subseteq X \setminus U\) such that \(A \in m(\{K_i : i \in I\})\). This implies that \(A\) is a \(\mathcal{U}_c\)-set. Since \(X\) is \(\mathcal{U}_c\)-admitting, we have \(A = \mathord{\downarrow} x\) for some \(x \in X\). Therefore \(\mathord{\downarrow} x \subseteq X \setminus U\), so \(x \notin U\). 

However, since \(A \cap K_i \neq \emptyset\) for every \(i \in I\), we have \(\mathord{\downarrow} x \cap K_i \neq \emptyset\), which implies \(x \in K_i\) for all \(i \in I\) because \(K_i\) is saturated. Hence, \(\mathord{\uparrow} x \in \mathcal{A}_i\) for all $i \in I$. Consequently,
\[
\mathord{\uparrow} x \subseteq \bigcap_{i \in I} \mathcal{A}_i \subseteq \square U,
\]
so \(\mathord{\uparrow} x \subseteq U\), which means \(x \in U\), contradicting \(x \notin U\). Therefore, there exists some \(i \in I\) such that \(\mathcal{A}_i \subseteq \square U\). We conclude that $P_\text{uc}(X)$ is $\mathcal{U}_c$-admitting.
\end{proof}

\section{The $\mathcal{U}_c$-completion of $T_0$ space}

In \cite{Li2017,Shen2019,gierz2003,Miao2024}, the \(D\)-completion, well-filterification, and soberification of \(T_0\) spaces were discussed. In \cite{LQ2025}, we have seen the reflection construction of \(\mathcal{U}_s\)-admitting spaces, which is in fact the \(\mathcal{U}_s\)-completion of \(T_0\) spaces. Naturally, we wonder what the \(\mathcal{U}_c\)-completion of a \(T_0\) space is.

\begin{definition}(\cite{Li2017,Shen2019,gierz2003})
    A \(D\)-completion (respectively, well-filterification, or soberification, or $\mathcal{U}_s$-complection) of a \(T_0\) space \(X\) consists of a \(d\)-space (respectively, well-filtered space, or sober space, or $\mathcal{U}_s$-admitting space) \(Y\) and a continuous map \(\eta: X \to Y\) which possesses the following universal property: for every continuous map \(f: X \to Z\) to a \(d\)-space (respectively, well-filtered space, or sober space, or $\mathcal{U}_s$-admitting space) \(Z\), there is a unique continuous map \(f^*: Y \to Z\) such that \(f = f^* \circ \eta\). That is to say, the following diagram is commutative.
\end{definition}
\[
\begin{tikzcd}[column sep=1.5cm, row sep=1.5cm]
X \ar[r, "\eta"] \ar[dr, "f"'] & Y \ar[d, "\exists \text{!}f^*"] \\
& Z
\end{tikzcd}
\]

\begin{definition}(\cite{Li2017,Xu2020,Keimel2009})
    A subset \(A\) of a \(T_0\) space \(X\) is called tapered (respectively, a well-filtered determined set, or $\mathcal{U}_s$-admitting determined set) if for any continuous mapping \(f: X \to Y\) to a \(d\)-space (respectively, well-filtered space, or $\mathcal{U}_s$-admitting space) \(Y\), there exists a unique \(y_A \in Y\) such that
\[
\overline{f(A)} = \overline{\{y_A\}}.
\]
The set of all closed tapered sets (respectively, closed well-filtered determined subsets, or $\mathcal{U}_s$-admitting determined set) of \(X\) is denoted by \(TP(X)\) (respectively, \(WD(X)\), or $USD(X)$).
\end{definition}

\begin{lemma}(\cite{Li2017,Xu2020,Keimel2009})
    Let \( X \) be a \( T_0 \) space. Let the family
\[
X^s = \{ A \subseteq X : A \in S(X) \}
\]
be topologized by the open sets
\[
U^d = \{ A \in X^s : A \cap U \neq \emptyset \}
\]
for each open subset \( U \) of \( X \). Define
\[
\eta: X \to X^s, \quad \eta(x) = \overline{\{x\}}.
\]
When \( S(X) = TP(X) \), the above construction gives the \( D \)-completion of \( X \); when \( S(X) = WD(X) \), it gives the well-filterification of \( X \); when \( S(X) = Irr(X) \), it gives the soberification of \( X \), and when \( S(X) = USD(X) \), it gives the  $\mathcal{U}_s$-complection of \( X \).
\end{lemma}

Likewise, next, we define the \(\mathcal{U}_c\)-completion of a \(T_0\) space and give its standard construction.
\begin{definition}
    Let \( X \) be a \( T_0 \) space and \( A \subseteq X \). The set \( A \) is called a \emph{\(\mathcal{U}_c\)-admitting determined set} (shortly, a \emph{\(UCD\)-set}) if for every \(\mathcal{U}_c\)-admitting space \( Y \) and every continuous map \( f: X \to Y \), there exists a unique \( y_A \in Y \) such that
\[
\overline{f(A)} = \overline{\{y_A\}}.
\]
The set of all closed \(\mathcal{U}_c\)-admitting determined subsets of \(X\) is denoted by \(UCD(X)\).
\end{definition}

\begin{definition}
     A $\mathcal{U}_c$-complection of a \(T_0\) space \(X\) consists of a $\mathcal{U}_c$-admitting space \(Y\) and a continuous map \(\eta: X \to Y\) which has the following universal property: for every continuous map \(f: X \to Z\) to a $\mathcal{U}_c$-admitting space \(Z\), there is a unique continuous map \(f^*: Y \to Z\) such that \(f = f^* \circ \eta\). That is to say, the following diagram is commutative.
\end{definition}
\[
\begin{tikzcd}[column sep=1.5cm, row sep=1.5cm]
X \ar[r, "\eta"] \ar[dr, "f"'] & Y \ar[d, "\exists \text{!}f^*"] \\
& Z
\end{tikzcd}
\]

Now we give the standard \(\mathcal{U}_c\)-completion of a \(T_0\) space \(X\) below.

\begin{theorem}
     Let \( X \) be a \( T_0 \) space. Let the family
\[
Y = \{ A \subseteq X : A \in UCD(X) \}
\]
be topologized by the open sets
\[
\Diamond U = \{ A \in Y : A \cap U \neq \emptyset \}
\]
for each open subset \( U \) of \( X \). Define
\[
\eta: X \to Y, \quad \eta(x) = \overline{\{x\}}=\mathord{\downarrow} x.
\]
This topology together with the map constitutes the standard \(\mathcal{U}_c\)-completion of \(X\).
\end{theorem}
\begin{proof}
    Claim 1: \(\eta\) is continuous.
 
In fact, for any open set \(U \subseteq X\),
$\eta^{-1}(\Diamond U) = \{ x \in X : \mathord{\downarrow} x \cap U \neq \emptyset \} = \{ x \in X : x \in U \} = U$.

    For every continuous map \(f: X \to Z\) to a $\mathcal{U}_c$-admitting space \(Z\), since \(A\) is a UCD-set, we have \(\overline{f(A)} = \overline{\{y_A\}}\). Define \(f^*: Y \to Z\) by \(f^*(A) = y_A\). 
 
 Cliam 2: \(f^*\) is the unique continuous map satisfying \(f^* \circ \eta = f\).

For any \(x \in X\), since \(f\) is continuous, we have
$\overline{f(\mathord{\downarrow} x)}$ = $\overline{f(x)}.$
Hence, $f^* \circ \eta(x) = f^*(\mathord{\downarrow} x) = f(x)$.
That is to say, \(f^* \circ \eta = f\).

For each \( V \in \mathcal{O}(Y) \),
\begin{align*}
(f^*)^{-1}(V)
&= \{ A \in \mathrm{UCD}(X) : f^*(A) \in V \} \\
&= \{ A \in \mathrm{UCD}(X) : \overline{f^*(A)} \cap V \neq \emptyset \} \\
&= \{ A \in \mathrm{UCD}(X) : \overline{f(A)} \cap V \neq \emptyset \} \\
&= \{ A \in \mathrm{UCD}(X) : f(A) \cap V \neq \emptyset \} \\
&= \{ A \in \mathrm{UCD}(X) : A \cap f^{-1}(V) \neq \emptyset \} \\
&= \Diamond f^{-1}(V).
\end{align*}
As $f^{-1}(V)$ is open in $X$,  $\Diamond f^{-1}(V)$ is open in Y. So $f^*$ is continuous.

In \(Y\), let $A$ be a UCD-set and let \(\Diamond U\) be an open set containing $A$, i.e., \(A \cap U \neq \emptyset\). Then there exists \(x \in X\) such that \(\mathord{\downarrow} x \in \Diamond U\). It implies that \(\overline{\eta(X)} = Y\), i.e., \(\eta(X)\) is dense in \(Y\).
Now suppose there exists a continuous map \(g: Y \to Z\) such that \(g \circ \eta = f\). Then for any \(x \in X\), $f(x) = g \circ \eta(x) = g(\mathord{\downarrow} x)$.
Since $\overline{f(\mathord{\downarrow} x)}$ = $\overline{f(x)}$, we have $f^*(\mathord{\downarrow} x) = f(x) = g(\mathord{\downarrow} x)$.
Thus \(f^*\) and \(g\) agree on the dense subset \(\eta(X)\) of \(Y\). By continuity, they must agree on all of \(Y\). As a result, \(g = f^*\). This proves the uniqueness of \(f^*\).

Claim 3: \( Y \) is a \(\mathcal{U}_c\)-admitting space.
 
 For any closed \(\mathcal{U}_c\)-set \(\mathcal{A} \subseteq Y\), there exists a closed set \( C \subseteq X \) such that
\[
\mathcal{A} = \square C = \{ A \in \mathrm{UCD}(X) : A \subseteq C \}.
\]
Because \(\mathcal{A}\) is closed, there exists an open set \( U \subseteq X \) such that \( Y \setminus \mathcal{A} = \Diamond U \). Let \( C = X \setminus U \). Then it is easy to verify that \(\mathcal{A} = \square C\). We claim that \( C \) is a UCD-set. To this end, let $f:X\to Z$ be any continuous map to a $\mathcal{U}_c$-admitting space $Z$. By Claim 2, there exists a unique continuous map $f^*:Y\to Z$ such that $f^*\circ \eta=f$.

Since $\mathcal{A}$ is a $\mathcal{U}_c$-set in $Y$ and $f^*$ is continuous, it follows from Proposition 3.12 that $f^*(\mathcal{A})$ is a closed $\mathcal{U}_c$-set in $Z$. As $Z$ is $\mathcal{U}_c$-admitting, every closed $\mathcal{U}_c$-set is the closure of a singleton. Thus there exists some $z_0\in Z$ such that
\[
\overline{f^*(\mathcal{A})}=\overline{\{z_0\}}.
\]

On the other hand,
\[
\overline{f^*(\mathcal{A})}=\overline{f^*(\Box C)}=\overline{f^*(\overline{\eta(C)})}=\overline{f^*(\eta(C))}=\overline{f(C)}.
\]
Therefore,
\[
\overline{f(C)}=\overline{\{z_0\}}.
\]
Hence, $C$ is a $\mathcal{U}_c$-admitting determined set, i.e., $C\in UCD(X)=Y$.

Consequently,
\[
\mathcal{A}=\Box C=\overline{\{C\}},
\]
where the last equality holds because in space Y, $\overline{\{C\}}=\{A\in UCD(X):A\subseteq C\}=\Box C$ whenever $C\in UCD(X)$.

Thus every closed $\mathcal{U}_c$-set in $Y$ is the closure of a singleton, i.e., $Y$ is a $\mathcal{U}_c$-admitting space.
\end{proof}

\bibliographystyle{elsarticle-num}
\bibliography{refs}

\end{document}